\documentclass[a4paper,11pt]{amsart}
\usepackage[T1]{fontenc}
\usepackage[english]{babel}
\usepackage[cp1252]{inputenc}
\usepackage{amsthm}
\usepackage{amsmath}
\usepackage{amsfonts}
\usepackage{stmaryrd}
\usepackage{amssymb}
\usepackage{hyperref}
\usepackage{indentfirst}
\usepackage{amssymb}
\usepackage{eufrak}
\usepackage{mathrsfs}
\usepackage{xypic}
\usepackage[pdftex]{graphicx}

\def\J{{\mathop{\rm J}}}
\def\Jac{{\mathop{\rm Jac}}}
\newcommand{\Cr}{\mathscr C}
\newcommand{\Verte}{\textrm Vert}
\newcommand{\di}{\textrm{d}}
\newcommand{\R}{\mathbb R}
\newcommand{\C}{\mathbb C}
\newcommand{\Z}{\mathbb Z}
\newcommand{\Arg}{\mathrm{Arg}\,}

\newcommand{\Log}{\mathrm{Log}\,}
\newcommand{\A}{\mathscr A}
\newcommand{\coA}{co\mathscr{A}}
\newcommand{\Si}{\mathscr S}
\newcommand{\conj}{\rm conj}

\newtheorem{remark}{Remark}[section]

\newtheorem{theorem}{Theorem}[section]
\newtheorem{definition}{Definition}[section]
\newtheorem{proposition}{Proposition}[section]
\newtheorem{corollary}{Corollary}[section]
\newtheorem{lemma}{Lemma}[section]

\begin{document}
\title{Analytic varieties with finite volume amoebas are algebraic}
\author{Farid Madani and Mounir Nisse}
\date{}
\address{NWF I-Mathematik, Universit\"at Regensburg, D-93040 Regensburg, Germany.}
\email{\href{mailto:Farid.Madani@mathematik.uni-regensburg.de}{Farid.Madani@mathematik.uni-regensburg.de}}

\thanks{The first author was supported by Alexander von Humboldt Foundation and by the Mathematisches Forschungsinstitut Oberwolfach in 2011, through the program "Oberwolfach Leibniz Fellows.}
\address{Department of Mathematics, Texas A\&M University, College Station, TX 77843-3368, USA.}
\email{\href{mailto:nisse@math.tamu.edu}{nisse@math.tamu.edu}}
\thanks{Research of  the second author is partially supported by NSF MCS grant DMS-0915245.}
\maketitle

\begin{abstract} In this paper, we study the amoeba volume of a given $k-$dimensional generic analytic variety $V$ of the complex algebraic torus $(\C^*)^n$. When $n\geq 2k$, we show that $V$ is algebraic if and only if the volume of its amoeba is finite. In this precise case, we establish a comparison theorem for the volume of the amoeba and the coamoeba. Examples and applications to the $k-$linear spaces will be given.
\end{abstract}



\section{Introduction}

    Some fundamental questions concerning the complex logarithm  lead us to  study certain  mathematical objects 
called amoebas and coamoebas, which are natural projections of complex varieties. They  have a strong  relations
 to several other areas of mathematics  as  real algebraic geometry,  tropical geometry, complex analysis, mirror 
symmetry, algebraic statistics and in several other areas. 
 Amoebas degenerate to a piecewise-linear object called {\em tropical varieties}, (see \cite{M1-02}, \cite{M2-04},
 \cite{M3-00}, \cite{FPT-00}, \cite{NS-09}, \cite{PR-04}, and \cite{PS-04}).
The behavior of a coamoeba at infinity is the so called the {\em phase limit set} of the variety itself.
 For a $k$-dimensional complex algebraic variety, the phase limit set contains an arrangement of $k$-dimensional flat
torus, which plays a crucial role on the geometry and the topology of both amoeba and coamoeba.
Moreover, these objects  are used  as an intermediate link between the classical and the tropical geometry. In this paper, 
we underline that the geometry of the amoeba  affects the algebraic structure of the variety.

It was shown by Passare and Rullg\aa rd \cite{PR-04} that the area of complex algebraic plane curve amoebas are finite.  In \cite{MN-11},  we proved that the volume of the amoeba of  a generic $k$-dimensional algebraic variety  of the complex algebraic torus $(\mathbb{C}^*)^{n}$ with $n\geq 2k$, is finite. 
There now arises the reciprocal question: Let $\mathcal{I}$ be an ideal in the ring of holomorphic function on $\mathbb{C}^n$, and $\overline{V}$ be the set of zeros of $\mathcal{I}$ which we assume generic of dimension $k\leq \frac{n}{2}$. If we  assume that the volume of the amoeba of $V = \overline{V}\cap (\mathbb{C}^*)^{n}$ is finite, then,  is $V$ algebraic? Theorem  \ref{main theorem} gives an affirmative answer to this question. 

\newpage

\noindent The {\it Main theorem} of this paper is the following:

\vspace{0.2cm}

\begin{theorem}\label{main theorem}
Let $V$ be a generic $k$-dimensional analytic variety in $(\C^*)^n$ with $n\geq 2k$. The following assertions are equivalent: \\ 
(i) The variety $V$ is algebraic;\\
(ii) The volume of $\A(V)$ is finite.
\end{theorem}


This paper is organized as follows. In Section \ref{curve section}, using the geometry and the combinatorial structure of the logarithmic limit and the phase limit sets, we prove Theorem \ref{main theorem} in the special case of curves. The proof of this case is actually the crucial step to show the general one. In Section \ref{proof main}, we show
 Theorem \ref{main theorem} for varieties of higher dimensions. In Section \ref{examples}, we give some examples of plane and spatial complex curves, underlying the importance of the finiteness of their amoeba areas. In Section \ref{comparison vol}, we prove a comparison theorem, asserting that up to rational number, the amoeba volume is bounded above and below by the coamoeba volume. Finally, in Section \ref{k-plan}, we give an application of the comparison theorem to the $k-$dimensional affine linear spaces in $(\C^*)^{2k}$. We compute the amoeba volumes of  $k-$dimensional  real affine linear  spaces in $(\C^*)^{2k}$.
 
\section{Preliminaries}\label{prelim}

Let $W$ be a complex variety in $\mathbb{C}^n$ defined by an ideal $\mathcal{I}$ of holomorphic functions  on $\mathbb{C}^n$. We say that a subvariety $V$  of the complex algebraic torus  $(\mathbb{C}^*)^n$ is analytic if there exists a complex variety $W$ as above such that $V:=W\cap (\mathbb{C}^*)^n$. All the analytic varieties considered in this paper are defined as above.
The {\em amoeba}  $\A$ of $V$ is by definition (see M. Gelfand, M.M. Kapranov
 and A.V. Zelevinsky \cite{GKZ-94}) the image of $V$ under the map :
\[
\begin{array}{ccccl}
\Log&:&(\mathbb{C}^*)^n&\longrightarrow&\mathbb{R}^n\\
&&(z_1,\ldots ,z_n)&\longmapsto&(\log |z_1|,\ldots ,\log|
z_n|).
\end{array}
\]
It is well known that the amoeba of a variety of codimension one  is closed and its complement components in $\mathbb{R}^n$ are convex (see \cite{FPT-00} ).
A. Henriques  gives an analogous definition for the convexity of amoeba complements  of higher codimension varieties as follows:
 \begin{definition}
A subset $A\subset \mathbb{R}^n$ is called $l$-convex if for any $l$-plane $L \subset \mathbb{R}^n$ the induced homomorphism $H_{l-1}(L\setminus A) \longrightarrow H_{l-1}(\mathbb{R}^n\setminus A)$ is injective.
\end{definition} 
Also, he proves that if  $V$ is a variety of codimension $l$ and $L$ is an $l$-plane of rational slope and $c$ is a non-zero $(l-1)$-cycle  in $H_{l-1}(L\setminus \mathscr{A}) $   then its image in $H_{l-1}(\mathbb{R}^n\setminus \mathscr{A})$ is non-zero, and then $\mathbb{R}^n\setminus \mathscr{A}$ is $l$-convex (see \cite{H-03} Theorem 4.1).
  
\vspace{0.1cm}

Let $V$ be an analytic variety as above. We denote by $\mathscr L^\infty (V)$ its logarithmic limit set which is the boundary of the closure of $r(\A(V))$ in the $n-$dimenstional ball $B^n$, where $r$ is the map defined by (see Bergman \cite{B-71}):
\[
\begin{array}{ccccl}
r&:&\mathbb{R}^n&\longrightarrow&B^n\\
&&x&\longmapsto&r(x) = \frac{x}{1+|x|}.
\end{array}
\]     
If $V$ is  an algebraic variety of dimension $k$, then its {\ logarithmic limit set} is a finite rational spherical polyhedron of dimension $k-1$ (i.e., a finite union of finite intersections of closed hemispheres and can be described in terms of  a finite number of inequalities with integral coefficients). More precisely,  we have  the following theorem structure  (see \cite{B-71} and \cite{BG-84}):

\begin{theorem}
[Bergman, Bieri-Groves] The logarithmic limit set $\mathscr{L}^{\infty}(V)$ of an algebraic variety $V$ in $(\mathbb{C}^*)^n$ is a finite union of rational spherical polyhedrons. The maximal dimension of a polyhedron in this union is achieved at least by one polyhedron $P$ in this union, and we have $\dim_{\mathbb{R}}\mathscr{L}^{\infty}(V) = \dim_{\mathbb{R}}P = (\dim_{\mathbb{C}}V) - 1$.
\end{theorem}

The argument map is the map  defined as follows:
\[
\begin{array}{ccccl}
\Arg&:&(\mathbb{C}^*)^n&\longrightarrow&(S^1)^n\\
&&(z_1,\ldots ,z_n)&\longmapsto&(\arg (z_1),\ldots ,\arg (
z_n) ).
\end{array}
\]
where $\arg (z_j) = \frac{z_j}{|z_j|}$. The {\em coamoeba} of $V$,  denoted by $co\mathscr{A}$, is its image under the argument map (defined for the first time by Passare on 2004).  On 2009, Sottile and the second author \cite{NS-09} define the {\em phase limit set}  of $V$, $\mathscr{P}^{\infty}(V)$ , as the set of accumulation points of arguments of sequences in $V$ with unbounded logarithm. If $V$ is an algebraic  variety of dimension $k$,  $\mathscr{P}^{\infty}(V)$  contains an arrangement of $k$-dimensional real sub-torus.

\section{Comparison between Amoeba and Coamoeba volumes}\label{comparison vol}

For a given map $f$, we denote by $\Jac(f)$ the Jacobian matrix of $f$, and by $\J(f)$ the determinant of $\Jac(f)$ when it exists.

\begin{proposition}\label{jacobian det}
Let $V$ be a $k$-dimensional complex submanifold in $(\C^*)^n$. The maps $\Log$ and $\Arg$ are well defined on $V$ and  
\begin{equation*}
 \partial \Log=\frac{1}{\Arg}\partial\Arg,\quad   \bar\partial \Log=\frac{-1}{\Arg}\bar\partial\Arg, \quad \bigl|\J(\Log_I)\bigr|=\bigl|\J(\Arg_I)\bigr|,
\end{equation*} 
for any $I\subset\{1,\ldots,n\}$ with cardinal $2k$.
\end{proposition}

\begin{proof}
We denote by $\{z_j\}_{1\leq j\leq n}$ the complex coordinates on $ \C^n$ and by $\{t_j\}_{1\leq j\leq k}$ the complex coordinates on $V$ given by a local chart $(\Omega,f)$ (i.e. $\forall z\in\Omega,\; t_j=f_j(z)$), where $\Omega$ is an open set of $V$ and $f$ is holomorphic from an open set of $\C^n$ to $\C^k$. Since $V$ is a complex submanifold of $(\C^*)^n$, the injection map $\imath: V\hookrightarrow (\C^*)^n$ is holomorphic. 
By definition,  for any $z\in V$ we have $\imath (z)=e^{Log z}\Arg z$. Since $\imath$ is holomorphic, $\bar\partial \imath (z)=0$ for any $z\in V$ (i.e. $\forall j\leq k,\; \partial_{\bar t_j} \imath (z)=0$ ). It implies that for any $j=1,\ldots,k$ and $z\in \Omega$
\begin{gather*}
\partial_{\bar t_j} \Log (z) =-\frac{1}{\Arg(z)}\partial_{\bar t_j} \Arg(z), \quad\partial_{t_j}\Log(z) =\frac{1}{\Arg(z)} \partial_{t_j}\Arg(z).
\end{gather*}
where the second equality holds by conjugating the first one. The statement of the proposition follows.
\end{proof}

\begin{theorem}\label{comparison theorem}
 Let $V$ be an analytic variety of $(\C^*)^n$ of dimension $k\leq \frac{n}{2}$. Let $\A$, $co\A$ be the amoeba and coamoeba of $V$ respectively. We suppose that $\Log: V\rightarrow \A$ and $\Arg: V\rightarrow co\A$ are locally finite coverings. We define two rational numbers $$p=\frac{\displaystyle\min_{\theta\in co\A\setminus\Arg S}\#\Arg^{-1}\{\theta\}}{\displaystyle\max_{y\in\A\setminus\Log S}\#\Log^{-1}\{y\}},\quad P=\frac{\displaystyle\max_{\theta\in co\A\setminus\Arg S}\#\Arg^{-1}\{\theta\}}{\displaystyle\min_{y\in\A\setminus\Log S}\#\Log^{-1}\{y\}}.$$ 
 Then 
 \begin{equation*}
 p\,vol(co\A)\leq vol(\A)\leq P\,vol(co\A).
\end{equation*}
In particular the volume of $\A$ is finite.
\end{theorem}

As an application of Theorem \ref{comparison theorem}, we have the following result that the authors had already proven in \cite{MN-11}:
\begin{corollary}\label{finite cor}
The amoeba of a $k$-dimensional generic complex algebraic variety in $(\C^*)^n$, with $2k\leq n$ has a finite volume. 
\end{corollary}
\begin{proof}
If $V$ is a generic complex algebraic variety in $(\C^*)^n$ of dimension $k\leq\frac{n}{2}$, then $\Log: V\rightarrow \A$ and $\Arg: V\rightarrow co\A$ are locally finite coverings. Therefore, the statement follows from Theorem \ref{comparison theorem}.
\end{proof}
\begin{proof}[Proof of Theorem \ref{comparison theorem}]
We have $V=V_{reg}\cup V_{sing}$ where $V_{reg}$ is the regular part of $V$ which is a $k$-dimensional complex submanifold in $\C^n$ and $V_{sing}$ is the singular part of $V$ which is a analytic subset of pure dimension less or equal than $k-1$. By Sard's theorem, the $2k$-measure of $\Log (V_{sing})$ and $\Arg (V_{sing})$ is zero. Hence, without loss of generality, we can suppose that $V=V_{reg}$ is a $k$-dimensional complex submanifold endowed with the induced metric $\imath^*\mathcal E_{2n}$, where $\mathcal E_{2n}$ is the standard Euclidean metric on $\C^n$ and $\imath:V\rightarrow\C^n$ is the injection map. Let 
\begin{equation*}
S=\{z\in V\; \vert\; \textrm{rank}\;\Jac\,\Log_z <2k\}
\end{equation*}
be the set of critical points of $\Log$. Using Proposition \ref{jacobian det}, we conclude that  the Jacobian matrix of $\Log$ and $\Arg$ have the same rank at any point in $V$. Hence, the set of critical points of $\Arg$ is also $S$. Using again Sard's theorem, the $2k$-measure of $\Log S$ and $\Arg S$ is zero. It yields  $vol(\A)=vol(\A\setminus\Log S)$ and $vol(co\A)=vol(co\A\setminus \Arg S)$.  The sets $\A\setminus\Log S$ and $co\A\setminus \Arg S$ are $2k$-dimensional real immersed submanifolds of $\R^n$ and $\mathbb T^n$ respectively. They are endowed with the induced metric $\imath^*\mathcal E_n$. Let $U_1$ and $U_2$ be two open sets in $V\setminus S$ such that $\Log:U_1\rightarrow \A\setminus\Log S$ and $\Arg:U_2\rightarrow co\A\setminus\Arg S$ are diffeomorphisms. By construction, we have the following identities: 

\begin{gather*}
vol(\A)=vol(\A\setminus\Log S,\imath^*\mathcal E_n)= vol(U_1,\Log^*\mathcal E_n),\\
vol(co\A)=vol(co\A\setminus\Arg S,\imath^*\mathcal E_n)= vol(U_2,\Arg^*\mathcal E_n).
\end{gather*}
Following the construction in \cite{MN-11} (see Section 3) and the identities above, the volume of $\A$ and $co\A$ are given by

\begin{gather}
vol(\A)=vol(U_1,\Log^*\mathcal E_n)=\int_{U_1} \bigl|\psi_{2k}\bigr|_{\Log^*\mathcal E_n}    \di v,\label{amoeba vol}
\\
vol(co\A)=vol(U_2,\Arg^*\mathcal E_n)=\int_{U_2} \bigl|\psi_{2k}\bigr|_{\Arg^*\mathcal E_n}    \di v\label{coamoeba vol}
\end{gather}
where $\di v$ is the restriction of the volume form of $\C^n$ to $V$, and $\psi_{2k}$ is $2k-$vector field in $\Lambda^{2k}TV$, such that $\di v(\psi_{2k})=1$. In a local complex coordinates $\{t_j\}_{1\leq j\leq k}$, the volume form $\di v$ and $\psi_{2k}$ are given by

\begin{equation*}
\di v=i^k\di t\wedge\di\bar t,\qquad \psi_{2k}=(-i)^k\frac{\partial}{\partial t}\wedge \frac{\partial}{\partial{\bar t}}
\end{equation*}
where $\di t=\di t_1\wedge\cdots \wedge\di t_k $ and $\frac{\partial}{\partial t}=\frac{\partial}{\partial t_1}\wedge\cdots \wedge\frac{\partial}{\partial t_k}$. It yields that 
\begin{gather*}
 \bigl|\psi_{2k}\bigr|_{\Log^*\mathcal E_n}^2=\bigl|\frac{\partial\Log}{\partial t}\wedge \frac{\partial\Log}{\partial{\bar t}}\bigr|_{\mathcal E_n}^2=\sum_{I\subset\{1,\ldots,n\},\#I=2k}\bigl|\textrm{J}(\Log_I)\bigr|^2\\ 
\bigl|\psi_{2k}\bigr|_{\Arg^*\mathcal E_n}^2=\bigl|\frac{\partial\Arg}{\partial t}\wedge \frac{\partial\Arg}{\partial{\bar t}}\bigr|_{\mathcal E_n}^2=\sum_{I\subset\{1,\ldots,n\},\#I=2k}\bigl|\textrm{J}(\Arg_I)\bigr|^2
\end{gather*}
where,  $\Log_I(z)=(\log |z_{i_1}|,\ldots,\log |z_{i_{2k}}|)$ for all $z\in V$ and $I=\{i_1,\ldots,i_{2k}\}$, and $\textrm{J}(\Log_I)$ is the Jacobian determinant of $\Log_I$ with respect to $\{t_j,\bar t_j\}_{1\leq j\leq k}$. We have from Proposition \ref{jacobian det}, $\bigl|\textrm{J}(\Log_I)\bigr|=\bigl|\textrm{J}(\Arg_I)\bigr|$ for any $I\subset\{1,\ldots,n\}$ with cardinal $2k$. We deduce that 
\begin{equation*}
\bigl|\psi_{2k}\bigr|_{\Log^*\mathcal E_n}=\bigl|\psi_{2k}\bigr|_{\Arg^*\mathcal E_n}.
\end{equation*} 
It remains to compare between  the volume of  $\A$ and $co\A$, given by \eqref{amoeba vol}, \eqref{coamoeba vol}. Since $\Log: V\rightarrow \A$ and $\Arg: V\rightarrow co\A$ are locally finite coverings, 
\begin{gather*}
\frac{1}{M_1}\int_{V\setminus S} \bigl|\psi_{2k}\bigr|_{\Log^*\mathcal E_n}  \di v\leq vol(\A)\leq \frac{1}{m_1}\int_{V\setminus S} \bigl|\psi_{2k}\bigr|_{\Log^*\mathcal E_n}    \di v,\\
\frac{1}{M_2}\int_{V\setminus S} \bigl|\psi_{2k}\bigr|_{\Arg^*\mathcal E_n}  \di v\leq vol(co\A)\leq \frac{1}{m_2}\int_{V\setminus S} \bigl|\psi_{2k}\bigr|_{\Arg^*\mathcal E_n}    \di v,
\end{gather*}
where $M_1=\displaystyle\max_{y\in\A\setminus\Log S}\#\Log^{-1}\{y\}$, $m_1=\displaystyle\min_{y\in\A\setminus\Log S}\#\Log^{-1}\{y\}$, \\   $M_2=\displaystyle\max_{\theta\in\A\setminus\Arg S}\#\Arg^{-1}\{\theta\}$ and $m_2=\displaystyle\min_{\theta\in\A\setminus\Arg S}\#\Arg^{-1}\{\theta\}$. We conclude that 
\begin{equation*}
\frac{m_2}{M_1}vol(co\A)\leq vol(\A)\leq \frac{M_2}{m_1}vol(co\A)
\end{equation*}
\end{proof}

\section{Analytic curves with finite  area amoebas are algebraic}\label{curve section} 
\vspace{0.2cm}

The purpose of this section is to prove the following:

\begin{theorem}\label{thm curve}
 Let $\mathscr C\subset (\C^*)^n$ be an analytic curve defined by an ideal of holomorphic functions 
$\mathcal I (\mathscr C)$ with $n\geq 2$. The amoeba area  of $\mathscr C$ is finite if and only if
 $\mathscr C$ is algebraic. 
\end{theorem}

In order to prove Theorem \ref{thm curve}, we start by proving a series of three lemmas.  

\begin{lemma}\label{rational lmm}
 Let $\mathscr C$ be a curve as in Theorem \ref{thm curve}. The logarithmic limit set 
$\mathscr L^\infty (\mathscr C)$ has the following properties: 

\begin{itemize}
\item[(i)]  Each point $s\in \mathscr L^\infty (\mathscr C)$ has a rational slope (i.e., there exists 
$\lambda\in \R$ such that $\lambda\cdot \overrightarrow{Os}\in\Z^n $);
\item[(ii)] The logarithmic limit set $\mathscr L^\infty (\mathscr C)$ is a finite set.
\end{itemize}
\end{lemma}

\begin{proof}
(i) Assume on the contrary that $\mathscr L^\infty (\mathscr C)$ contains a point $s$ such that the vector
 $\overrightarrow{Os}$ has an irrational slope. Let $s=(u_1,\ldots,u_n)$,  and 
\begin{equation*}
\begin{split}
 & \R\longrightarrow \R^n\\
x & \longmapsto (u_1x+a_1,\ldots, u_nx+a_n) 
\end{split}
\end{equation*}
be the parametrization of the asymptotic straight  line  $D_s$ in $\R^n$ to the amoeba $\A(\mathscr C)$, in direction $s$.  Under this assumption, the phase limit set $\mathscr P^\infty (\mathscr C)$ contains a subset
 of dimension at least 2. Indeed, there exists an affine line in the universal covering of the real torus 
$(S^1)^n$, parametrized by $y\mapsto (u_1y_1+b_1,\ldots, u_ny_n+b_n)$, such that the closure of its projection in 
$(S^1)^n$ is a 2-dimensional torus $T_s$, and is contained in the phase limit set $\mathscr P^\infty (\mathscr C)$.
  It implies that there exists a regular open subset $U\subset co\A(\mathscr C)$ which is covered infinitely many
 times under the argument map. Namely, $\Arg^{-1}(U)=\cup_{i=1}^\infty V_i$, where $V_i$ are open, regular and disjoint
 sets in $\Cr$.  In fact, let $U$ be a regular open subset of $co\mathscr{A}(\mathscr{C})$ contained in $T_s$. 
Let $\tilde{V}_1$ be a connected component of $\Arg^{-1}(U)$ and $V_1 = \Log (\tilde{V}_1)$.  Let $I_2$ be a 
segment in $D_s\setminus \overline{V}_1$ such that $\Arg (\Log ^{-1}(I_2))$ intersect $U$. Indeed, $I_2$ exists 
because the immersed circle  of slope $s$ in the real torus $(S^1)^n$ is dense in $T_s$. Let $\tilde{V}_2$ be the 
connected component of $\Arg (\Log ^{-1}(U))$ such that $\Log (\tilde{V}_2)$ contains $I_2$, and
 $V_2 = \Log (\tilde{V}_2)$. By construction, $V_1\cap V_2 $ is empty.  We do the same thing for $V_2$, we
 obtain then an infinite sequence of open subsets $V_1, V_2, \ldots$ in the amoeba such that: $V_i\cap V_j = \phi$ 
for each $i\ne j$, and  the area of $V_i$ is equal to the area of $U$ for each $i$ (this comes from the fact 
that the absolute value of the determinant  of the jacobian of $\Log\circ \Arg^{-1}$ is equal to one, which means
 that locally the map $\Log\circ\Arg^{-1}$ conserves the volume).  This is a contradiction with the finiteness of the area of the amoeba. Hence, $s$ cannot be irrational.  

\vspace{0.2cm}


(ii) If  $\mathscr L^\infty (\mathscr C)$ is not finite, then it contains an accumulation point, since $S^{n-1}$
 is compact. Let $s\in\mathscr L^\infty (\mathscr C)$ be an accumulation point. Namely, there exists a sequence of 
points $s_m\in\mathscr L^\infty (\mathscr C)$ such that $\lim_{m\to +\infty}s_m=s$, where each $s_m$ has a 
rational slope, and up to a multiplication by a real number,
 $\overrightarrow{Os}_m=(u_1^{(m)},\ldots,u_{n}^{(m)})$, with $u_j^{(m)}=\frac{a_j^{(m)}}{b_j^{(m)}}$,
 where $a_j^{(m)}$ and $b_j^{(m)}$ are integers. \\
The sequence $\{ s_m\}$ converges and is not stationary. This means that there exists $1\leq j\leq n$ such that the
 sequence $\{ b_j^{(m)}\}$ is unbounded. The circles corresponding to that slopes,  have length greater or equal to $2\pi b_j^{(m)}$. Hence, if $U$ is a regular subset of measure different than zero in the coamoeba, then it is covered infinitely many times under the argument map since the $b_j^{(m)}$ tend to infinity. This fact contradicts the fact that the area of the amoeba is finite. So, $\mathscr L^\infty (\mathscr C)$  is finite.

\end{proof}

\begin{lemma}\label{finite lmm}
 Let $\mathscr C\subset (\C^*)^n$ be a generic  analytic curve such that the area of its amoeba is finite, and
 $s$ be a point in $\mathscr{L}^{\infty}(\mathscr{C})$. The number of ends $\mathscr{E}$ of the amoeba 
$\mathscr{A}(\mathscr{C})$, such that $\overline {r(\mathscr{E})}\cap S^{n-1} = \{ s\}$, is finite.
\end{lemma}

\begin{proof}
Assume on the contrary that the  number of ends $\mathscr{E}$ of the amoeba $\mathscr{A}(\mathscr{C})$, such that $\overline {r(\mathscr{E})}\cap S^{n-1} = \{ s\}$, is infinite. Hence, two cases can hold.
\noindent The first case: the number of corresponding circles of $s$ in $(S^1)^n$ is finite. Then, there exists at 
least a circle $C$ covered  by the argument mapping infinitely many times. Take a regular open  
$U\subset co\mathscr{A}(\mathscr{C})$ bounded by a segment in $C$. Then, the number of connected components of 
$\Arg^{-1}(U)$ is infinite, and $\Log (\Arg^{-1}(U))$ has also an infinite number of connected components with the 
same area as $U$. This is a contradiction with the assumption.

\noindent The second case: the number of corresponding circles of $s$ in $(S^1)^n$ is infinite. Then, there exists a 
circle $C$ which is an accumulation circle with the same slope $s$, because the real torus is compact. Using the
 same reasoning as above, we obtain a contradiction. 
 
This implies that the number of ends of the amoeba with the same slope is finite.

\end{proof}

\begin{lemma}\label{henriques lmm}
 Let $\mathscr C$ be a generic curve in $(\C^*)^n$. If $S^{n-2}$ is a sub-sphere of $S^{n-1}=\partial B^n$, invariant under the involution $-\rm{id}$, then  
$\mathscr L^\infty (\mathscr C)$ intersects the interior of each connected components of
 $S^{n-1}\setminus S^{n-2}$. 
\end{lemma}

\begin{proof}
Using Henriques theorem \cite{H-03}, we know that the complement components of amoebas are $l-$convex, if
 ${\rm codim} V=l+1$. Lemma \ref{henriques lmm} is a consequence of this fact. If $\mathscr C$  is generic
 (it is not contained in a complex algebraic torus of smaller dimension), then the intersection of the closed half 
spaces in $\mathbb{R}^n$  bounded by the hyperplanes normal to the directions 
$s\in \mathscr{L}^{\infty}(\mathscr{C})$ is compact.
\end{proof}

Let $\mathscr C\subset (\C^*)^n$ be a generic analytic curve with defining ideal $\mathcal I(\mathscr C)$ such 
that the area of its amoeba is finite. Let $s\in\mathscr L^\infty (\mathscr C)$ with slope $(u_1,\ldots,u_n)$,
 and $D_s$ be a straight line in $\R^n$ directed by $s$ and asymptotic to the amoeba $\mathscr{A}$. We denote  
by $\mathscr{H}({D_s})$  the holomorphic cylinder  which is the lifting of $D_s$ and asymptotic to the end of 
$\mathscr C$ corresponding to $s$.

\begin{lemma}\label{decomposition lmm}
Let $\mathscr C\subset (\C^*)^n$ be a generic analytic curve as above.
 Then, the ideal $ \mathcal I(\mathscr C)$ is generated by a set of holomorphic functions $\{f_1,\ldots,f_q\}$ such that for any $j=1,\ldots, q$, we 
have $f_j=h_jg_j$, where $h_j$ are holomorphic functions  without  zeros in a neighborhood of 
$\mathscr{H}({D_s})$,
  and $g_j$  are entire functions where the terms of their power series have powers in the closed  half space 
 $\{ \alpha\in\R^n \, \vert\,  \sum_{i=1}^n\alpha_i.u_i\leq 0\}$.
\end{lemma}

\begin{proof}
Without loss of generality, by Lemma \ref{finite lmm}, we can  assume that  the amoeba contains only one end corresponding to  $s\in \mathscr L^\infty (\mathscr C)$.
If we don't have such  decomposition of the $f_j$'s, then $s$ is an accumulation point, and this cannot happen 
because the area of the amoeba is finite by assumption. So, there exist $h_j$, $g_j$, and $U(s)$ is a neighborhood
 of a holomorphic cylinder on which $h_j$ don't vanish. Namely, the holomorphic cylinder is the lifting in 
$(\C^*)^n$ via the logarithmic map of a straight line in $\R^n$ of slope $(u_1,\ldots,u_n)$. This holomorphic 
cylinder is such that the closure in $B^n$ of the retraction by $r$ (defined in Section \ref{prelim}) of 
the image under the
logarithmic map  of one of  its ends, intersects the boundary of $B^n$ precisely on $s$. The function $g_j$ are entire 
where all their terms (as power series) are of powers $\alpha$ with $\sum u_i\alpha_i\leq 0$. Without loss of 
generality, we can assume that $\Cr$ is an irreducible curve. Hence, $\Cr$ is contained in the zero locus of 
the $g_j$'s. In other words, our curve $\Cr$ is contained in the curve with defining ideal $\mathcal{I}_{s}$ spanned the by $g_j$'s.
\end{proof}


\begin{proof}[Proof of Theorem \ref{thm curve}]

If the area of the amoeba of  $\Cr$ is finite, then using Lemma \ref{henriques lmm}, and doing the same thing 
for each vertex in $\mathscr L^\infty (\mathscr C)$ as in Lemma \ref{decomposition lmm}   we obtain polynomials $p_1,\ldots,p_q$ 
such that $\Cr$ is contained in $\Cr_r$ where $\Cr_r$ is the algebraic curve defined by the ideal generated by
 $p_1,\ldots,p_q$. 
Hence, $\Cr$ is an irreducible component of $\Cr_r$, then  it is algebraic. 
 If the curve is algebraic then using Corollary \ref{finite cor}, we deduce that the area of the amoeba 
$\mathscr{\Cr}$ is finite because the area of the coamoeba is always finite. 
\end{proof}

\begin{corollary}
Let $\mathscr{C}$ be an analytic curve (not necessary algebraic) in $(\mathbb{C}^*)^n$.
 Then $\mathscr{L}^{\infty}(\mathscr{C})$ is the union of a finite number of isolated points with rational slopes 
and a finite number of  a geodesic arcs with rational  end slopes. In particular, if $\mathscr{C}$ is not
 algebraic, then the number of arcs in $\mathscr{L}^{\infty}(\mathscr{C})$ is different than zero.
\end{corollary}

\vspace{0.2cm}

This is a consequence of the  proof of Lemma \ref{rational lmm}. The arcs are geodesic (i.e., contained in some circle invariant under the involution $x\mapsto -x$) because   the contrary means that the phase limit set contains a 
flat torus of dimension at least three, and this contradict the fact that the dimension of the coamoeba  is equal 
 two.

\vspace{0.2cm}

\begin{corollary}\,\,
Let $f$ be an entire function in two variables. Then $f = hp$ where $h$ is a holomorphic function which doesn't vanish in the complex algebraic torus and $p$ is a polynomial if and only if the area of amoeba of the  holomorphic
 curve in $(\mathbb{C}^*)^2$ defined by $f$ is finite.
\end{corollary}

\vspace{0.1cm}
\section{Proof of the Main theorem}\label{proof main}
\vspace{0.3cm}

In this section, we generalize the result of Section \ref{curve section} to k-dimensional analytic varieties in 
the complex algebraic torus $(\C^*)^n$ with $n\geq 2k$. The following lemma generalizes Lemma 
\ref{rational lmm} to higher dimensions.

\begin{proposition}\label{finite proposition}
 Let $V$ be a $k-$dimensional  generic analytic variety in $(\C^*)^n$, with $n\geq 2k$. Assume that the volume of
 its amoeba is finite. Then its logarithmic limit set $\mathscr L^\infty (V)$ is a finite rational spherical polyhedron of dimension $k-1$.
\end{proposition}
   
\begin{proof}
 If $\mathscr L^\infty (V)$ is not a rational spherical polyhedron, and $v$ is a vertex of $\mathscr L^\infty (V)$
 with irrational slope $u$, then the phase limit set  $\mathscr P^\infty(V)$ contains a torus of dimension at
 least  $(k-1)+2=k+1$ which is  the closure of an immersed circle of slope $u$. Using the same argument as in  the 
curve  case (i.e., the fact that locally $| \J(\Log\circ \Arg^{-1})|=1$,  the map $\Log\circ \Arg^{-1}$
  conserves the volumes), we deduce that 
   the volume of $\A(V)$ in infinite. Indeed, if $U$ is a regular subset of the coamoeba $co\mathscr{A}(V)$ with nonvanishing
 volume, then it is covered via the argument map infinitely many times. This is in contradiction with the 
assumption on the volume of the amoeba  $\mathscr{A}(V)$. Therefore, $\mathscr L^\infty (V)$ is rational. 

Now, suppose that 
$\dim (\mathscr L^\infty (V)) > k-1$. Then, it contains a point of
 irrational slope in some $l-$cell with $l\geq k$. Hence, its phase limit set contains a torus of dimension at 
least  $l+1\geq k+1$.  By the same argument as in the proof of Lemma \ref{rational lmm}, we obtain a contradiction
 on the assumption on the volume of the amoeba.

 It remains to show that  $\mathscr L^\infty (V)$ is finite.  Otherwise, it contains 
an accumulation $(k-1)$-cell, which means that $\mathscr{P}^{\infty}(V)$ contains a torus of dimension at least 
$k+1$ (where $k+1 = (k-1) +1 +1$, such that $k-1$ comes from the dimension of the cells, the first one comes from 
the circle corresponding to the direction $s$, and the second one comes from the accumulation direction).  Using 
the same reasoning as in Lemma \ref{rational lmm} (ii), we get a contradiction on  the assumption of  the amoeba
 volume.
\end{proof}

\vspace{0.2cm}

\begin{definition}
 Let $s$ be a vertex of $\mathscr L^\infty (V)$. A polyhedron $P$ is  in the direction of $s$, if for any point 
$x\in P$, there exists a vector $v_x$ in $P$ with starting point $x$ and  slope $s$.
\end{definition}

\vspace{0.2cm}

The following lemmas are the analogous  of Lemma \ref{finite lmm} and Lemma \ref{decomposition lmm}
in higher dimension and their  proofs are similar.

\vspace{0.2cm}

\begin{lemma}\label{hfinite lmm}
 Let $V \subset (\C^*)^n$ be a generic  $k$-dimensional analytic variety with defining ideal $\mathcal{I}(V)$ 
such that $n\geq 2k$, and  the volume of its amoeba is finite. Let $s\in \Verte (\mathscr L^\infty (V))$ be a 
vertex, and $\Sigma_s$ be the open sub-complex of $\mathscr L^\infty (V)$ with only one vertex $s$. Then, there is
 a finite number of  polyhedrons $P$ of dimension $k$, asymptotic to the amoeba $\mathscr{A}(V)$ in the direction of $s$, and such that $\overline{r(P)}\cap S^{n-1}\subset \Sigma_s$.
\end{lemma}

\begin{proof}
If the number of  polyhedrons $P$ of dimension $k$, asymptotic to the amoeba $\mathscr{A}(V)$ in  direction  $s$, and such that $\overline{r(P)}\cap S^{n-1}\subset \Sigma_s$,  is infinite, then there exists a sequence of parallel polyhedrons $\{P_m\}$  satisfying the same property. Hence, there are two possibilities: (i) the number of their corresponding $k$-dimensional real torus in $(S^1)^n$ is finite, which means that at least one of them is covered under the argument map infinitely many times. This contradict the fact that the volume of the amoeba is finite. (ii) the number of their corresponding $k$-dimensional real torus in $(S^1)^n$ is  infinite, which means that they contain at least one accumulation $k$-torus (of course parallel to all of them). This is a contradiction with the assumption on the volume of the amoeba.
\end{proof}


\vspace{0.2cm}

Let $V \subset (\C^*)^n$ be a   generic $k$-dimensional analytic variety with defining ideal 
$\mathcal I(V)$ such that the volume of its amoeba is finite.
 Let $s$ be a vertex of $\mathscr L^\infty (V)$ with slope $(u_1,\ldots,u_n)$,
and $\mathscr{P}_s$ be the finite set of polyhedrons in the direction of $s$, and asymptotic to the amoeba $\mathscr{A}(V)$.
We denote  
by $\mathscr{H}({\mathscr{P}_s})$   the  union of the holomorphic cylinders  which are the lifting of the 
polyhedrons in $\mathscr{P}_s$, and asymptotic to the ends of $V$ corresponding to $s$.

\begin{lemma}\label{hdecomposition lmm}
 Let $V \subset (\C^*)^n$ be a   generic $k$-dimensional analytic variety.
 Then,  the ideal $ \mathcal I(V)$ is generated by a set of holomorphic functions $\{f_1,\ldots,f_q\}$
 such that for any $j=1,\ldots, q$, we have $f_j=h_jg_j$, where $h_j$ are holomorphic functions  without  zeros in a neighborhood of $\mathscr{H}({\mathscr{P}_s})$,  and $g_j$  are entire functions where the terms of their power series have powers in the closed  half space $\{ \alpha \, |\,  \alpha . u\leq 0\}$  where $\alpha . u = \sum_{i=1}^n\alpha_i.u_i$.
\end{lemma}

\begin{proof}
The proof of this lemma is similar to the proof of Lemma \ref{decomposition lmm}.
\end{proof}

\vspace{0.3cm}

\begin{proof} [End  of  Theorem \ref{main theorem}  proof]  The second implication i.e., (ii) $\Rightarrow $ (i) is a consequence of
 Theorem \ref{comparison theorem}. Without loss of generality, we may assume that $V$ is irreducible. So,  (i) $\Rightarrow$ (ii) is a consequence of Proposition \ref{finite proposition}, Lemma \ref{hfinite lmm}, Lemma \ref{hdecomposition lmm}, and Lemma \ref{henriques lmm}.
\end{proof}

\vspace{0.2cm}
\section{Examples}\label{examples}
\vspace{0.2cm}

{\bf 1.} Let $\Cr$ be the complex plane  curve in $(\C^*)^2$ given by the zeros of the holomorphic function 
$f(z_1,z_2)=z_2-e^{z_1}$. Since the complex rank of $\Jac(f)$ is one, $\Cr$ is a Riemann surface. A parametrization
 of $\Cr$ is given by  $t\in \mathbb{C}^*\mapsto (t,e^t)\in (\mathbb{C}^*)^2$. The amoeba $\A(\Cr)$ is the set of points in $\R^2$ delimited by the graphs of the
 two functions $x\mapsto \pm e^x$ (see Figure \ref{exp}).   The set of critical points of $\Log$ restricted to
 $\Cr$ is $S=\{(x,e^x)\in(\R^*)^2\}$. The map $\Log$  is 2-sheets covering over its regular values. However, the 
map $\Arg$ is not a locally finite covering (i.e., $\#\Arg^{-1}\theta=+\infty$ for any regular value 
$\theta\in co\A$). Since the closure of the coamoeba is compact, its area is always finite. But, the amoeba has an
 infinite area.  This means that the assumption on Theorem \ref{comparison theorem} is necessary.\\
Moreover, we can check that the logarithmic limit set of $\Cr$ has 1-dimensional connected component and one 
isolated point of rational slope. The phase limit set of $\Cr$ is the whole torus.

\begin{figure}[h!]
\begin{center}
\includegraphics[angle=0,width=0.6\textwidth]{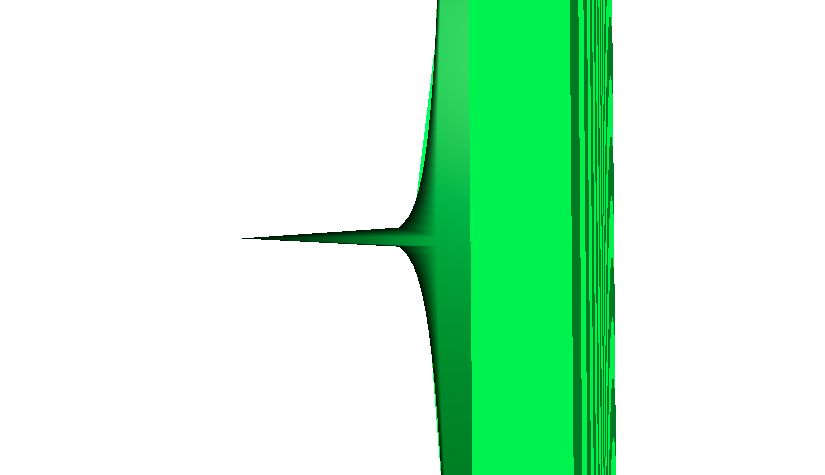}\qquad\qquad
\caption{The amoeba of the plane holomorphic  curve given by the parametrization $g(z)=(z,e^z)$.}
\label{exp}
\end{center}
\end{figure}

\vspace{0.3cm}

{\bf 2.} Let $\Cr$ be the complex plane curve in $(\C^*)^2$ parametrized by 
\begin{equation*}
 \begin{split}
\rho: D &\longrightarrow (\C^*)^2\\
  t &\longmapsto (\cos t, \sin t)
 \end{split}
\end{equation*}

 where  $D=\{t=a+ib \, \vert\, (a,b)\in (]0,2\pi[\setminus\{\frac{\pi}{2},\pi,\frac{3\pi}{2}\})\times \R  \}$ is the fundamental domain. We can check that $\Cr$ is contained in the algebraic curve with defining polynomial $p(x,y)=x^2+y^2-1$. On the other hand, we know that the last curve is irreducible. Hence, $\Cr$ is algebraic and defined by the same polynomial. The critical points of $\Log$ are $\rho((]0,2\pi[\setminus\{\frac{\pi}{2},\pi,\frac{3\pi}{2}\})\times\{0\})$. The logarithmic limit set consists of three points with coordinates $(-1,0)$, $(0,-1)$ and $(\frac{\sqrt 2}{2},\frac{\sqrt 2}{2})$. The phase limit set consists of three pairs of circles with distinct slopes given by $-\infty$, 0 and 1.
   
\vspace{0.2cm}

{\bf 3.} Let $\Cr$ be the complex plane curve in $(\C^*)^2$ parametrized by 
\begin{equation*}
 \begin{split}
\rho: D &\longrightarrow (\C^*)^2\\
  t &\longmapsto (z,e^z,z+1)
 \end{split}
\end{equation*}

The logarithmic limit set of $\Cr$ is the union of two points and an arc in the sphere $S^2$. Its phase limit set is an arrangement of two circles and 2-dimensional flat torus. The curve $\Cr$ is not algebraic and the area of its amoeba (see Figure 2) is infinite.

\begin{figure}[h!]
\begin{center}
\includegraphics[angle=0,width=0.8\textwidth]{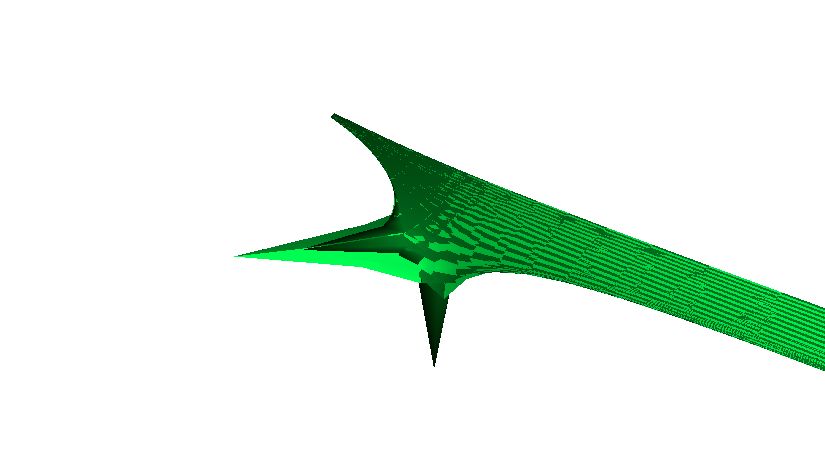}\qquad\
\caption{The amoeba of the spatial holomorphic  curve given by the parametrization $g(z)=(z,e^z,z+1)$.}
\label{exp Spatial}
\end{center}
\end{figure}

\vspace{0.2cm}

\section{Amoebas of $k-$dimensional affine linear spaces}\label{k-plan}

\vspace{0.1cm}

Let $k$, and $s$ be two positives integers, and $\mathscr{P}(k)\subset (\mathbb{C}^*)^{k+s}$ be the  affine
linear space of dimension  $k$ given by the  parametrization $\rho$:
\[
\begin{array}{ccccl}
\rho&:&(\mathbb{C}^*)^k&\longrightarrow&(\mathbb{C}^*)^{k+s}\\
&&(t_1,\ldots ,t_k)&\longmapsto&(t_1,\ldots ,t_k,f_1(t_1,\ldots ,t_k),\ldots ,f_s(t_1,\ldots ,t_k)),
\end{array}
\]
such that  $f_j(t_1,\ldots ,t_k) = b_j+\sum_{i=1}^ka_{ji}t_i$ where  $a_{ji}$ and $b_j$ are  complex numbers for
 $i=1,\ldots , k$  and $j=1,\ldots , s$. First of all, we can assume that
 $f_1(t_1,\ldots ,t_k) = 1+\sum_{i=1}^kt_i$. Moreover, we assume that the affine linear spaces are in general position. 
\begin{definition} 
Let $V\subset (\mathbb{C}^*)^n$ be an algebraic subvariety, and $\conj$ be the  involution on $\mathbb{C}^n$ given by conjugation on each coordinate. We say the variety  $V$ is real if it is invariant under $\conj$, i.e., $\conj (V) = V$. The real part of  $V$ denoted by $\mathbb{R}V$  is the set of points in $V$ fixed by $\conj$.
\end{definition}

In this section, we assume  $n=2k+m$, where $m$ is a nonnegative integer (i.e., $s=k+m$). The goal of this section is to
 prove the following theorem:

\begin{theorem}\label{k-plan thm}
Let $\mathscr{P}(k)$ be a generic linear space in $(\mathbb{C}^*)^{2k+m}$, and $\theta$ is a regular value of the 
argument map. Then, the cardinality of $\Arg^{-1}(y)$ is equal to one.
\end{theorem}

Before we give the proof of this theorem, let us  start by looking to  the case $k=1$, and the following remark:

\vspace{0.2cm}

\begin{remark}  {\rm
Without loss of generality, we can assume that $\mathscr{P}(1)$ is the line in $(\mathbb{C}^*)^{2+m}$ given by the parametrization:
\[
\begin{array}{ccccl}
\rho&:&\mathbb{C}^*&\longrightarrow&(\mathbb{C}^*)^{2+m}\\
&&t_1&\longmapsto&(t_1,f_1(t_1),\ldots ,f_{1+m}(t_1)),
\end{array}
\]
such that  $f_1(t_1) = 1 + t_1$, and 
$f_j(t_1) = b_j+ a_{j1}t_1$, where  $a_{j1}$ and $b_j$ are  complex numbers for $j=2,\ldots , 1+m$. Let $y = (y_1,\ldots, y_{2+m})$ be a regular value of the logarithmic map. The number of point $t_1\in \mathbb{C}^*$ such that $|t_1| = e^{y_1}$, and $|f_1(t_1)| = e^{y_2}$ is at most two (because the intersection of two circles cannot exceed two points; the first circle has a center at the origin and radius $e^{y_2}$ and the second has the center at $(1,0)$ and radius $e^{y_1}$). It is clear that if $y$ is regular, then that  number  is equal two. Otherwise, the circles are tangent, and then $y$ is critical. Indeed,  if we make  a small perturbation of $t_1$ in some direction, the point goes out of the amoeba, which means that $y$ is in the boundary of the amoeba.  Moreover, we can check 
 that these two points are conjugate. 
 So, if the number is equal two, then $|f_j(t_1)| $ should be equal to  $|f_j(\overline{t}_1)|$ for any $j=1,\ldots, 1+m$. This means that $| b_j+ a_{j1}t_1| = | b_j+ a_{j1}\overline{t}_1|$, or equivalently $\cos (\arg (a_{j1}) + \arg
  (t_1) - \arg (b_j))  = \cos (\arg (a_{j1}) + 2\pi - \arg (t_1) - \arg (b_j))$. Hence, $\arg (a_{j1}) - \arg (b_j) = 0 \mod (\pi )$, which is equivalent to $\frac{a_{j1}}{b_j}\in \mathbb{R}^*$ for any $j$, and then the curve is real. Otherwise, the cardinality of $\Log^{-1}(y)$ is equal one. }
  
\end{remark}

\vspace{0.2cm}

\begin{proof}[Proof of Theorem \ref{k-plan thm}]
Let $\theta = (\theta_1,\ldots ,\theta_k, \psi_1,\ldots ,\psi_{k+m})$ be a regular value of the argument map. 
If $\rho(t)\in \mathscr{P}(k)$ belongs to the inverse image by the argument map of $\theta$, then $\arg (t_j) = \theta_j$ 
for each $j=1,\ldots ,k$, and $\arg(f_l(t)) =\psi_l$ for each $l=1,\ldots , k+m$. The $k$-plane $\mathscr{P}(k)$
 is parametrized by $\rho$ as above.  For each $l$,  we can always view $f_l$, $b_l$, and the $a_{lj}t_j$'s as vectors in the plane $\mathbb{C}$, such that their arguments are in the  increasing order.
If we  put them in the plane with this order, we obtain a convex polygon. 
We can check that if there exist  $t\ne t'$ with $\rho (t)$ and $\rho (t')$  in $\mathscr{P}(k)$, and 
$\Arg (\rho (t)) = \Arg (\rho (t')) = \theta$,
then the inverse image of $\theta$ by $\Arg$ has dimension strictly greater than zero. Therefore, $\theta$ in not regular.  In fact, if $|t_j| \ne |t'_j|$ for some $j$, then for any 
$(\lambda |t_j| + (1-\lambda )|t'_j|)e^{\theta_j}$, with $\lambda\in [0,1]$, there exist $t_s=\mu_s e^{\theta_s}$ with
 $\mu_s\in \mathbb{R}_+$,\, $s\in \{1,\ldots ,k\}\setminus \{ j\}$, and $\Arg (f_l)=\psi_l$ for each 
$l=1,\ldots , k+m$. This means that the dimension of $\Arg^{-1}(\theta )$ is at least one.
We conclude that the cardinality of $\Arg^{-1}(\theta)$ is equal one for any regular point $\theta$.
\end{proof}

\vspace{0.2cm}

In the case of $k$-dimensional real linear spaces of  $(\mathbb{C}^*)^{2k}$,
 the second author, Johansson and Passare  \cite{JNP-10} prove that each  regular value of the logarithmic map is covered by $2^k$ points, and the volume of the coamoeba is equal to $\pi^{2k}$. Hence, the numbers $p$ and $P$ of  Theorem \ref{comparison theorem}, are equal  to $\frac{1}{2^k}$  in this case. 
 
 \begin{corollary}
 Let $\mathscr{A}(k)$ be the amoeba of a $k$-dimensional real linear space of $(\mathbb{C}^*)^{2k}$. Then, 
  $$
   vol(\A) =  \frac{\pi^{2k}}{2^k}.
  $$.
\end{corollary}

\bibliographystyle{amsalpha}
\bibliography{biblioM}

\end{document}